\documentclass[11pt]{amsart}

\usepackage{colonequals}
\usepackage[alphabetic]{amsrefs}
\usepackage{graphicx}
\usepackage{fullpage}
\usepackage{array}
\usepackage{amssymb}
\usepackage{amsmath}
\usepackage{fancyhdr}
\usepackage{parskip} 
\usepackage{float}
\usepackage{xcolor} 
\usepackage{enumerate} 
\usepackage{geometry}

\newtheorem{theorem}{Theorem}
\newtheorem*{theorem1}{Theorem 1}
\newtheorem*{theorem2}{Theorem 2}
\newtheorem{lemma}[theorem]{Lemma}

\theoremstyle{definition}

\theoremstyle{remark}
\newtheorem{remark}[theorem]{Remark}

\numberwithin{equation}{section}

\makeatletter
\@namedef{subjclassname@2020}{%
  \textup{2020} Mathematics Subject Classification}
  \makeatother

\usepackage{hyperref}

\definecolor{urlcolor}{rgb}{0,.145,.698}
\definecolor{linkcolor}{rgb}{.71,0.21,0.01}
\definecolor{citecolor}{rgb}{.12,.54,.11}
\definecolor{outerrorbackground}{HTML}{FFDFDF}

\hypersetup{breaklinks=true, colorlinks=true, urlcolor=urlcolor, linkcolor=linkcolor, citecolor=citecolor}

\begin{document}

\title{A Fair Shake: How Close Can The Sum Of $n$-Sided Dice Be To A Uniform Distribution?}
\author{Shamil Asgarli}
\email{sasgarli@scu.edu}
\author{Michael Hartglass}
\email{mhartglass@scu.edu}
\author{Daniel Ostrov}
\email{dostrov@scu.edu}
\author{Byron Walden} 
\email{bwalden@scu.edu}
\address{Department of Mathematics and Computer Science \\ Santa Clara University \\ 500 El Camino Real \\ USA 95053}
\maketitle

\begin{abstract}
\noindent Two possibly unfair $n$-sided dice, both labelled $1, 2, \ldots, n$, are rolled, and the sum is recorded. How should the dice's sides be weighted so that the resulting sum is closest to the uniform distribution on $2, 3, \ldots, 2n$? We answer this question by explicitly identifying the optimal pair of dice. This resolves a question raised by Gasarch and Kruskal in 1999 in a surprising way. We present additional results for the case of more than two possibly unfair $n$-sided dice and for the hypothetical case where the weights on each die are permitted to be negative, but must still sum to one.

\end{abstract}

\section{Introduction}

If you roll two six-sided dice, the probability distribution for their sums is triangle shaped:
\begin{figure}[h!]
\centering
\includegraphics[scale=.25, trim=0 0 0 18]{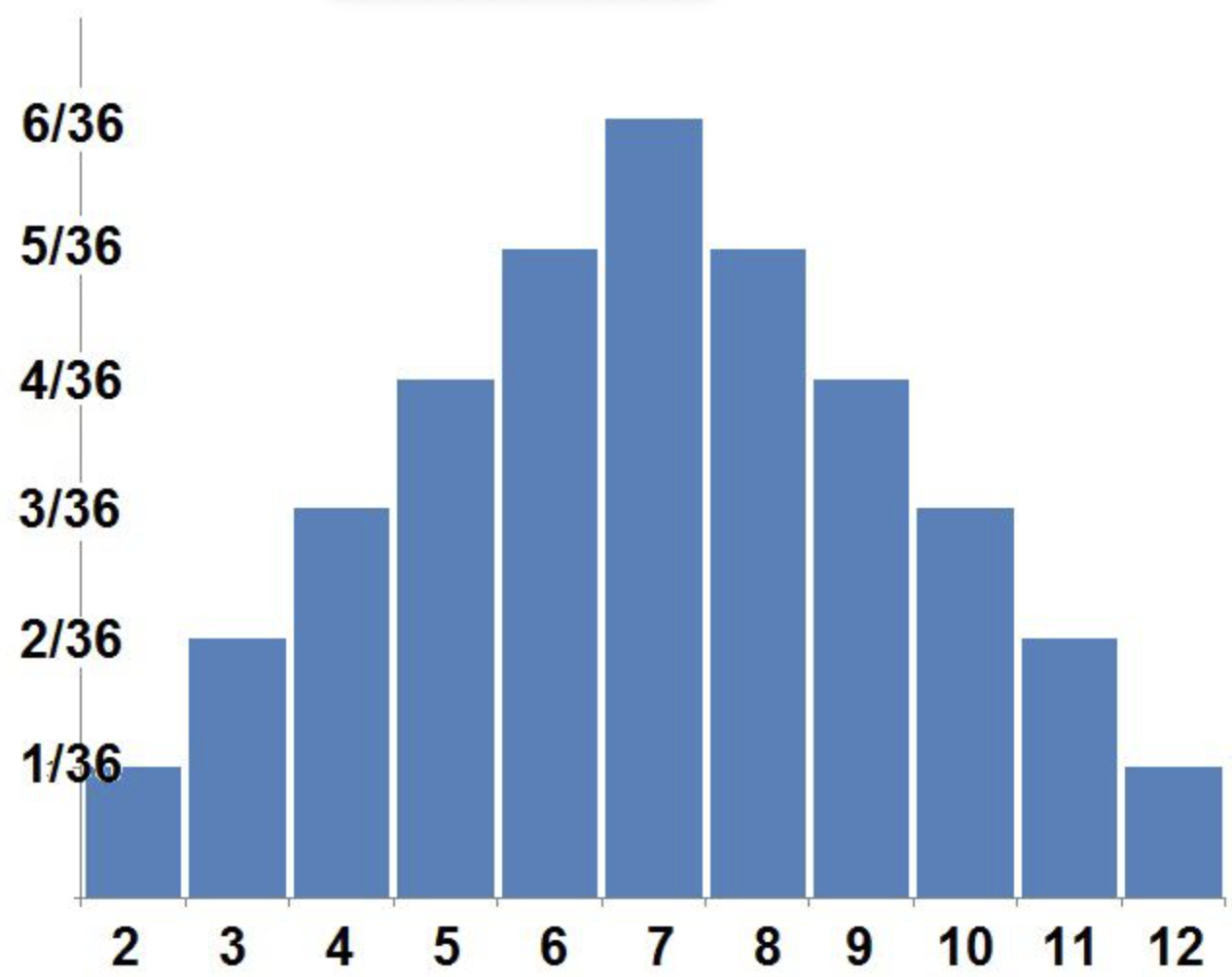}
\label{Fig1} 
\end{figure}\\
This, of course, assumes that there is an equal chance of obtaining each of the six sides. If you could change the probability of rolling each of the six sides for each of the two dice, could you obtain a uniform distribution for each sum?  That is, could you replace the triangular distribution above with a rectangular distribution?

This question was posed by John Kelly in 1950 \cite{Kelly1}, and subsequently it was proven via a number of methods that no matter how the two dice are weighted, a uniform distribution cannot be obtained; see, for example, \cite{Kelly2}, \cite{Hon}, or \cite{Hof}. In particular, a recent book by Bollob\'{a}s \cite{Bol} (see Problem 11) contains two solutions: one uses an elementary argument involving the arithmetic mean-geometric mean (AM-GM) inequality, and the other solution employs generating functions. Further, this was extended by Chen-Rao-Shreve \cite{Chen-Rao-Shreve} to show that even if you have $m$ dice, each with $n$ sides, the dice cannot be weighted so that a uniform distribution is attained for the dice sums. More recently, Morrison~\cite{Mor}, following the main results from Gasarch and Kruskal ~\cite{Gasarch-Kruskal}, systematically constructed when $m$ weighted dice --- each with a potentially different numbers of sides, unlike in \cite{Chen-Rao-Shreve} and this paper --- lead to the dice sums being uniformly distributed. 
% More recently, Morrison~\cite{Mor}, following the main results from Gasarch and Kruskal ~\cite{Gasarch-Kruskal}, systematically constructed finite sets of dice with \emph{nonstandard} labelling where the dice sums \emph{are} uniformly distributed. 

Given the result from Chen-Rao-Shreve \cite{Chen-Rao-Shreve} that a uniform distribution cannot be attained, Gasarch and Kruskal pose an interesting question at the end of their article ~\cite{Gasarch-Kruskal} that we will focus on in this paper: what weightings for each of two $n$-sided dice will minimize $D$, the sum of the squared difference between the probability for each dice sum and the uniform distribution? We note that this minimum $D$ must be positive, since $D=0$ corresponds to the unattainable uniform distribution. 

Gasarch and Kruskal state that by programming in Matlab, they find that ``[f]or all $n$, Matlab produced symmetric dice that were identical to each other," where the term ``symmetric dice" means that the probability of rolling a 1 and rolling an $n$ is the same, the probability of rolling a 2 and rolling an $n-1$ is the same, etc., and the term ``identical to each other" means that the probability of rolling any specific side is the same for both dice. For example, for two six-sided dice, they find that Matlab minimized $D$ when the probability of rolling a 1 or a 6 on either die is $0.243883$, the probability of rolling a 2 or a 5 is $0.137480$, and the probability of rolling a 3 or a 4 is $0.118637$. They are, however, careful to state that ``Matlab {\it does not} guarantee that the results are the true optimum, so the question of whether or not optimal dice must be identical and symmetric is interesting and open, even in the cases where we obtained numerical results."

In this paper, we show that Gasarch and Kruskal were right to carefully point out that their Matlab evidence was not the end of the story. In fact, their Matlab results do \emph{not} accurately reflect the true minimum of $D$ when $n>2$! Specifically, we will prove the following theorem.
\begin{theorem}\label{thm:2-dice}
The two $n$-sided dice with side probabilities $\left(\frac{1}{2}, 0, 0,\cdots,0, 0, \frac{1}{2}\right)$ and $\left(\frac{2}{3n-2}, \frac{3}{3n-2}, \frac{3}{3n-2},\cdots,\frac{3}{3n-2},\frac{3}{3n-2}, \frac{2}{3n-2}\right)$ are, up to swapping, the unique pair that minimize $D$, the sum-of-squares difference from the uniform distribution. This minimized value of $D$ is $\frac{1}{2(2n-1)(3n-2)}$, which corresponds to the following probabilities for each of the dice sums: $\left(\frac{2}{2(3n-2)}, \frac{3}{2(3n-2)}, \frac{3}{2(3n-2)},\cdots,\frac{3}{2(3n-2)}, \frac{3}{2(3n-2)}, \frac{4}{2(3n-2)}, \frac{3}{2(3n-2)}, \frac{3}{2(3n-2)}, \cdots, \frac{3}{2(3n-2)}, \frac{3}{2(3n-2)}, \frac{2}{2(3n-2)}\right)$.
\end{theorem}
According to Theorem~\ref{thm:2-dice}, the actual minimum of $D$ corresponds to symmetric dice, but if $n>2$, they are \emph{not} identical to each other. For example, when $n=3$, the closest to uniform distribution is obtained when one of the two dice has a half and half chance of rolling a 1 or a 3 and the other die has a $\frac{2}{7}$ chance of rolling a 1, a $\frac{3}{7}$ chance of rolling a 2, and a $\frac{2}{7}$ chance of rolling a 3. 

Another potential way to obtain a uniform distribution is to allow for the possibility of negative weights to be assigned to some dice sides. That is, we still require that the weights assigned to the sides of each die add up to one, but we allow these side weights to take negative values. When does this allow us to attain a uniform distribution? The answer is given by the following theorem that we will prove in this paper:
\begin{theorem}\label{thm:neg-prob}
If we have $m$ dice where $m\ge 2$ and we allow negative weights to be assigned to sides of dice, then weightings can be found to attain a uniform distribution (i.e., $D=0$) if and only if $n$, the number of dice sides on each of the $m$ dice, is an odd number.
\end{theorem}

In Section~\ref{sect:notation}, we introduce notation needed for our proofs. We establish Theorem~\ref{thm:2-dice} and Theorem~\ref{thm:neg-prob} in Section~\ref{sect:optimal-2-dice} and Section~\ref{sect:pqneg}, respectively. In the spirit of Gasarch and Kruskal, in Section~\ref{sect:multiple-dice}, which concludes this paper, we present two open conjectures supported by what appears to be strong numerical evidence. The easily modified Python program used to gather this numerical evidence is available in an online appendix at \url{https://tinyurl.com/294ycnha}.

% ---------------------------------------------

\section{Notation} \label{sect:notation}

In the sections that follow, we will use the following notation:

\begin{itemize}

    \item $m$ is the number of dice rolled.

    \item $n$ is the number of sides on each of the $m$ dice. The sides are labelled $1, 2,..., n$. In general, we will use the letter $i$ to index the sides. That is, $i = 1,2,...,n$.

    \item $p_i$ is the probability (or weight) corresponding to the first die landing on side $i$. Similarly, $q_i$ is the probability (or weight) corresponding to the second die landing on side $i$. In Section~\ref{sect:multiple-dice}, we will also use $r_i$ for the result of a third die. In Section~\ref{sect:pqneg} and the next bullet point, we will discuss $m$ dice, so we will use the notation $p_{i_1},q_{i_2},...,z_{i_m}$ to denote the probability of landing on each side of the first die, the second die,..., the $m^{\mbox{th}}$ die.

    \item $c_j$ is the probability that the sum of the $m$ dice equals $j$. So, for example, for two standard dice (i.e., $m=2, n=6,$ and $p_i=q_i=\frac{1}{6}$ for $i = 1,2,...,6$), we have $c_3 = \frac{2}{36}$, since there are two ways to roll a 3 (rolling 1 then 2 or rolling 2 then 1). In general, we will use the letter $j$ to index the possible dice sums. That is, if $m=2$, we have $j=2,3,...,2n$, and for a general $m$, we have $j=m,m+1,...,mn$. Given this, for $m$ dice, $c_j$ is defined by $c_j=\sum_{i_1 + i_2 +\cdots+i_m=j}p_{i_1} \times q_{i_2}\cdots \times z_{i_m}$ where $j=m,m+1,...,mn$.

    \item For two dice, $D$, the sum of the squared difference between the probability for each dice sum and the uniform distribution, is given by $D = \sum_{j=2}^{2n}\left(c_j -\frac{1}{2n-1}\right)^2$. For $m$ dice, $D = \sum_{j=m}^{mn}\left(c_j -\frac{1}{m(n-1)+1}\right)^2$.  Note that for two dice, the uniform distribution corresponds to $c_j =\frac{1}{2n-1}$ for $j=2,3,...,2n$, while for $m$ dice, it corresponds to $c_j =\frac{1}{m(n-1)+1}$ for $j=m,m+1,...,mn$.

\end{itemize}

\section{Identifying the optimal dice} \label{sect:optimal-2-dice}

In this section, we will prove Theorem~\ref{thm:2-dice}, which establishes the probabilities for each side of two $n$-sided dice that are needed to minimize $D$, and therefore obtain the most uniform distribution for the sum of the dice.  We first rewrite Theorem~\ref{thm:2-dice} using the notation from the last section:
\begin{theorem1}
$D$ is minimized if and only if either 
\begin{enumerate}
    \item $p_i=0$, except $p_1=p_n=\frac{1}{2}$, and $q_i=\frac{3}{3n-2}$, except $q_1=q_n=\frac{2}{3n-2}$, or
    \item we swap the roles of all the $p_i$ and the $q_i$ above, which, by symmetry, yields the same $D$ value.
\end{enumerate}
This minimum possible value for $D$ is $\frac{1}{2(2n-1)(3n-2)}$, which corresponds to $c_j=\frac{3}{2(3n-2)}$, except $c_2=c_{2n}=\frac{1}{(3n-2)}$ and $c_{n+1}=\frac{2}{(3n-2)}$. 
\end{theorem1}

We next motivate our approach. Recall that $c_j$ is the probability of the dice summing to $j$, given that $p_i$ is the probability of rolling an $i$ with the first $n$-sided die and $q_i$ is the probability of rolling an $i$ with the second $n$-sided die, so $i=1,2,...,n$ and $j=2,3,...,2n$.  Since the $c_j$ are probabilities, we know that $\sum_{j=2}^{2n} c_j = 1$. Therefore, we can re-express $D$, the quantity that we are trying to minimize that measures the non-uniformity of the dice sum distribution, by
\begin{equation*}
  D=\sum_{j=2}^{2n} \left(c_j-\frac{1}{2n-1}\right)^2 
  = \sum_{j=2}^{2n} c_j^2 - 2\sum_{j=2}^{2n}\frac{c_j}{2n-1} +
  \sum_{j=2}^{2n}\left(\frac{1}{2n-1}\right)^2
  = \sum_{j=2}^{2n} c_j^2 - \frac{1}{2n-1}. 
\end{equation*} 
Therefore, in order to minimize $D$, it suffices to minimize $\sum_{j=2}^{2n} c_j^2$.

% We make a key observation for understanding why the uniform distribution $c_{2} = c_{3} = \cdots = c_{2n}$ is unattainable using an argument that is basically the first solution in Bollobás \cite{Bol} to Kelly's \cite{Kelly2} problem.  
We make a key observation for understanding why the uniform distribution $c_{2} = c_{3} = \cdots = c_{2n}$ is unattainable using an argument that follows both the method in the first solution in Bollobás \cite{Bol} and the approach of Moser and Wahab in \cite{Kelly2}. This observation centers on examining the relationship between $c_{2}, \, c_{n+1}$ and $c_{2n}$.  Note that all three of these probabilities depend on $p_{1}, \, q_{1}, \, p_{n}$, and $q_{n}$.   The key inequality relating $c_{2}, \, c_{n+1}$ and $c_{2n}$ is AM-GM, the arithmetic mean-geometric mean inequality:

\begin{lemma}\label{AM-GM}
$c_{n+1} \geq 2\sqrt{c_{2}c_{2n}}$ and equality holds if and only if $p_{i}q_{n+1-i} = 0$ for $2 \leq i \leq n-1$ and $p_{1}q_{n} = p_{n}q_{1}$.
\end{lemma}

\begin{proof}
%We first note that if any one of $p_{1}$, $p_{n}$, $q_{1}$, or $q_{n}$ are zero, then $c_{2}c_{2n} = 0$, so we trivially have $c_{n+1} \geq 2\sqrt{c_{2}c_{2n}}$.  Furthermore, as $c_{n+1} = \sum_{i=1}^{n}p_{i}q_{n+1-i}$, it follows that if $c_{n+1} = 2\sqrt{c_{2}c_{2n}} = 0$, then $p_{i}q_{n+1-i} = 0$ for all $i$ which in particular implies that $p_{1}q_{n} = p_{n}q_{1}$.
We have, by the definition of $c_{n+1}$ and AM-GM, that
$$ c_{n+1}  =  \sum_{i=1}^{n}p_{i}q_{n+1-i}  \geq  p_1 q_n + p_n q_1 \geq 2\sqrt{p_1 q_1 p_n q_n}=2\sqrt{c_2 c_{2n}}.
$$
In order for $c_{n+1} = 2\sqrt{c_2 c_{2n}}$, every inequality sign must be an equality. Recalling conditions for equality in AM-GM, we must have $p_{i}q_{n+1-i}= 0$ for $2 \leq i \leq n-1$ for the first equality and $p_1 q_n = p_n q_1$ for the second equality. \end{proof}

One now sees why it is impossible for all of the $c_{i}$'s to have the same value: requiring that $c_2=c_{n+1}=c_{2n}>0$ would, from Lemma \ref{AM-GM}, imply that $1\ge 2$!  
%This in particular implies that the $c_{i}$'s cannot %have the same value.  (commented out since it seems %to duplicate previous sentence. Perhaps something %else was intended. --blw)
Lemma \ref{AM-GM} also suggests a reasonable path forward.  Rather than try to immediately solve for the $p_{i}$ and $q_{i}$ that minimize $\sum_{j=2}^{2n}c^2_{j}$, we should use the fact that the $c_{j}$ are probabilities, in conjunction with Lemma \ref{AM-GM}, to directly find the $c_{j}$ that minimize $\sum_{j=2}^{2n}c_{j}^{2}$ and then use our optimal $c_{j}$ values to obtain the $p_{i}$ and $q_{i}$ values.

To gain a better understanding of the next step, we recall the Cauchy-Schwarz inequality in $\mathbb{R}^{k}$:
$$
|x_{1}y_{1} + x_{2}y_{2} + \cdots + x_{k}y_{k}| \leq \sqrt{x_{1}^{2} + x_{2}^{2} + \cdots + x_{k}^{2}}\cdot \sqrt{y_{1}^{2} + y_{2}^{2} + \cdots + y_{k}^{2}},
$$
where  $x_{1}, \cdots, x_{k}, y_{1}, \cdots, y_{k}$ are any real numbers.  Recall that equality in this expression holds if and only if the vectors $(x_{1}, \cdots, x_{k})$ and $(y_{1}, \cdots, y_{k})$ are scalar multiples of each other.  One immediate consequence of Cauchy-Schwarz is that for any real $x_{1}, \cdots, x_{k}$:
$$
x_{1}^{2} + \cdots + x_{k}^{2} \geq \frac{(x_{1} + \cdots + x_{k})^{2}}{k}, 
$$
which is obtained by setting $y_{1} = y_{2} = \cdots = y_{k} = 1$ and squaring both sides of the Cauchy-Schwarz inequality.

To this end, we separate out the quantities $c_{2}$, $c_{n+1}$ and $c_{2n}$ (as we have information on how they relate), and we let $\Omega$ represent the remaining indices; that is, $\Omega = \{2, 3, \cdots, 2n\} \setminus \{2, n+1, 2n\}$. If we let $s = c_{2} + c_{n+1} + c_{2n}$, then, by the Cauchy-Schwarz inequality, we have that
\begin{eqnarray*}
\sum_{j=2}^{2n} c_j^2 & = & c_2^2 + c_{2n}^2+c_{n+1}^2 + \sum_{j \in \Omega} c_j^2\\
& \geq & c_2^2 + c_{2n}^2+c_{n+1}^2 + \frac{1}{2n-4}\left(\sum_{j \in \Omega} c_j\right)^{2} = c_2^2 + c_{2n}^2+c_{n+1}^2 + \frac{1}{2n-4}(1-s)^{2}.
\end{eqnarray*}

Therefore, to minimize $\sum_{j=2}^{2n} c_j^2$, it is reasonable to try to find $c_{2}, c_{3}, \cdots, c_{2n}$ so that $c_{i} = c_{j}$ whenever $i, j \in \Omega$.  If we can also bound $c_2^2 + c_{2n}^2+c_{n+1}^2$ below by an expression in $s = c_{2} + c_{n+1} + c_{2n}$, then we have bounded $\sum_{j=2}^{2n} c_j^2$ below by an expression in $s$. To that end, we note that
for any $x, y, z\geq 0$, the Cauchy-Schwarz inequality yields $x^2+y^2+z^2 \geq \frac{1}{3}(x+y+z)^2$. For our purpose of finding optimal dice, we use the additional stipulation $z\geq 2\sqrt{xy}$ to obtain a stronger inequality in the following lemma:

\begin{lemma}\label{lemma:inequality-parabola} Let $x,y,z \geq 0$, with $x+y+z=k$ and $z\geq 2\sqrt{xy}$. Then $x^2+y^2+z^2 \geq \frac{3}{8}(x+y+z)^2$, with equality if and only if $x=y=\frac{k}{4}$ and $z= \frac{k}{2}$. 
\end{lemma}

\begin{proof}[Proof] We have $z^2-4xy\geq 0$ by hypothesis.  Some algebra shows that the following identity holds:
$$
8(x^2+y^2+z^2)-3(x+y+z)^2 = 2(z^2-4xy) + (z-x-y)^2 + (z-2x)^2 + (z-2y)^2 \geq 0.
$$
The desired inequality  $x^2+y^2+z^2\geq \frac{3}{8}(x+y+z)^2$ immediately follows. Moreover, equality requires that $z=2x$ and $z=2y$, which then forces $z^2=4xy$ and $z=x+y$. Since $x+y+z=k$, we conclude that the equality holds if and only if $x=y=\frac{k}{4}$ and $z=\frac{k}{2}$.
\end{proof}

In our situation, since we know that $c_{n+1} \geq 2\sqrt{c_{2}c_{2n}}$, then setting $x = c_{2}$, $y = c_{2n}$ and $z = c_{n+1}$ gives us the inequality $c_{2}^2 + c_{n+1}^2 + c_{2n}^2 \geq \frac{3}{8}s^{2}$ with equality if and only if $c_{2} = c_{2n}$ and $c_{n+1} = 2c_{2}$.  From Lemma \ref{AM-GM}, this means that in order for the equality $c_{2}^2 + c_{n+1}^2 + c_{2n}^2 = \frac{3}{8}s^{2}$ to hold, we must have $p_{i}q_{n+1-i} = 0$ whenever $2 \leq i \leq n-1$, and $p_{1}q_{n} = p_{n}q_{1}$.

Putting this all together, we now have that
\begin{eqnarray*}
\sum_{j=2}^{2n} c_j^2 & \geq & c_2^2 + c_{2n}^2+c_{n+1}^2 + \frac{1}{2n-4}(1-s)^{2} \\
& \geq & \frac{3}{8}s^{2} + \frac{1}{2n-4}(1-s)^{2},
\end{eqnarray*}
with equality if and only if $c_{i} = c_{j}$ whenever $i, j \in \Omega$, $p_{i}q_{n+1-i} = 0$ whenever $2 \leq i \leq n-1$, and $p_{1}q_{n} = p_{n}q_{1}$. 

The path forward is now clear:  We find $s \in [0, 1]$ that minimizes $f(s) = \frac{3}{8}s^{2} + \frac{1}{2n-4}(1-s)^{2}$. If we could guarantee that $\sum_{j=2}^{2n} c_j^2 = \frac{3}{8}s^{2} + \frac{1}{2n-4}(1-s)^{2}$, then from our conditions on equality in the previous paragraph, we could then explicitly solve for each $c_{j}$.  If we can find $p_{i}$'s and $q_{i}$'s that produce these $c_{j}$, then we have obtained our desired minimum.

Observe that the graph of $f(s) = \frac{3}{8}s^2 + \frac{1}{2n-4} (1-s)^2$ is an upward-facing parabola whose vertex is at the value of $s$ where $f'(s)= \frac{3}{4} s - \frac{1}{n-2} (1-s) = 0$. That value is $s = \frac{4}{3n-2}$, which gives us
\begin{equation*}
   \sum_{j=2}^{2n} c_j^2 \geq \frac{3}{8} \frac{16}{(3n-2)^2} + \frac{1}{2n-4} \frac{(3n-6)^2}{(3n-2)^2} = \frac{3}{2(3n-2)}. 
\end{equation*}

For equality to hold, $c_{2} = c_{2n} = \frac{s}{4} = \frac{1}{3n-2}$, $c_{n+1} = \frac{s}{2} = \frac{2}{3n-2}$ and if $j \in \Omega$, then $(2n-4)c_{j} = 1-s$. meaning that $c_j = \frac{1}{2n-4} (1 - s) = \frac{1}{2n-4} (1 - \frac{4}{3n-2}) = \frac{3}{2(3n-2)}$. We note that $\sum_{j=2}^{2n} c_j = 1$ with these values, so our $c_j$ values for $j=2,3,...,2n$ are the unique values that minimize $D$.

Recapping, we have that
\begin{equation*}
  D=\sum_{j=2}^{2n} \left(c_j-\frac{1}{2n-1}\right)^2 = \sum_{j=2}^{2n} c_j^2 - \frac{1}{2n-1}  \geq \frac{3}{2(3n-2)} - \frac{1}{2n-1} = \frac{1}{2(2n-1)(3n-2)},  
\end{equation*} 
where equality -- and therefore minimizing $D$ -- can be achieved with the optimal choices of $c_2=c_{2n}=\frac{1}{(3n-2)}$, $c_{n+1}=\frac{2}{(3n-2)}$, and $c_j=\frac{3}{2(3n-2)}$ for $j \in \Omega$.

To finish the proof of Theorem \ref{thm:2-dice}, we need to find the values of the $p_{i}$ and $q_{i}$ that will realize the values of $c_{j}$ just obtained above.

{\it Claim:} To attain the optimal values of $c_j$ above, we either have that $p_1=p_n=1/2$, $q_1=q_n=2/(3n-2)$, and, for each $i=2,\ldots,n-1$, $p_i=0$ and $q_i=3/(3n-2)$, or we swap the roles of the $p_i$ and $q_i$, so
$q_1=q_n=1/2$, $p_1=p_n=2/(3n-2)$, and, for each $i=2,\ldots,n-1$, $q_i=0$ and $p_i=3/(3n-2)$.

\begin{proof}[Proof of claim]
Recall that since $c_{n+1} = 2\sqrt{c_2 c_{2n}}$, we must have that $p_{i}q_{n+1-i}= 0$ for $2 \leq i \leq n-1$. In particular, it follows that $p_{2}q_{n-1} = 0$.  We will assume that $p_{2} = 0$, instead of $q_{n-1}=0$, and use this assumption to show that $p_{k}$ must be 0 for all $k \in \{2, \cdots, n-1\}$. This observation will then determine the values of each $q_{i}$.  If we assume that $q_{n-1} = 0$, instead of assuming $p_{2}=0$, a symmetric argument leads to the same results but with the values of the $p_i$ and the $q_i$ swapped. 

Since $c_2=p_1q_1$ and $c_{2n}=p_n q_n$, we get:
$$
p_1 q_1 + p_n q_n = c_2 + c_{2n} = c_{n+1} = p_1 q_n + p_n q_1,
$$
which can be rearranged as
$$
p_1 q_1 + p_n q_n - p_1 q_n - p_n q_1 = (p_1-p_n)(q_1-q_n)=0,
$$
meaning either $p_1=p_n$ or $q_1=q_n$. However, $p_1 q_1 = p_n q_n$ (since $c_2=c_{2n}$), so each of these equalities forces the other to hold, meaning 
$$p_1=p_n \mbox{ and } q_1=q_n.$$

Next, using that $p_2=0$, we look at $c_3=p_1q_2+p_2q_1 = p_1q_2$. Since $c_3 = \frac{3}{2}c_2=\frac{3}{2}p_1 q_1$, we obtain:
$
\frac{3}{2}p_1 q_1 = p_1q_2.
$
which implies that $q_{2} = \frac{3}{2}q_{1}$.

 Recall that we must have that $p_{i}q_{n+1-i}= 0$ for $2 \leq i \leq n-1$. In particular, setting $i=n-1$, we have that $p_{n-1}q_2 = 0$ and so, since $q_2> 0$, we have that $p_{n-1}=0$. 

Recapping, we have that $p_1=p_n \mbox{ and } q_1=q_n$, which we will use at the end of our proof, and we have that $p_2=p_{n-1}=0$ and $q_2=\frac{3}{2}q_1>0$, which will form the initial step for the following induction argument:

\begin{quote}

{\bf Induction step:} For $i\in\{1,2,...,n-3\}$, if we have that $p_{n-1}=p_{n-2} = ... =p_{n-i} = 0$, then $q_{n-i} = \frac{3}{2}q_1$ and $p_{n-i-1}=0$.

{\it Proof of the induction step:} Combining $c_{2n-i}=\frac{3}{2}c_2 = \frac{3}{2} p_1q_1$ and the definition of $c_{2n-i}$, we have:
$$
\frac{3}{2} p_1q_1 = c_{2n-i} = p_{n} q_{n-i} + p_{n-1} q_{n-i+1} + ... + p_{n-i} q_n.
$$
Given that $p_{n-1}=p_{n-2}=...=p_{n-i}=0$, the above equality implies that $\frac{3}{2} p_1q_1 = p_n q_{n-i} = p_1 q_{n-i}$, which means $q_{n-i}=\frac{3}{2} q_1$. 

On the other hand, combining $c_{n-i+1}=\frac{3}{2}c_2 = \frac{3}{2} p_1q_1$ and the definition of $c_{n-i+1}$, we have:
$$
\frac{3}{2} p_1q_1 = c_{n-i+1} = p_{1} q_{n-i} + p_{2} q_{n-i-1} + ... + p_{n-i-1} q_{2} + p_{n-i}q_{1}.
$$
Since $q_{n-i}=\frac{3}{2}q_1$, we see that $p_{1} q_{n-i}$ term above is already $\frac{3}{2}p_1q_1$, which means the other terms, being non-negative, must each be zero. In particular, $p_{n-i-1}q_2 = 0$, which forces $p_{n-i-1}=0$ since $q_2> 0$.
\qed

\end{quote}

With the induction step established, we can run it using $i=1$, then $i=2$, $i=3$, and so on, until we stop at $i=n-3$. This establishes that $p_i=0$ and $q_i=\frac{3}{2}q_1$ for $i = 2,3,...,n-1.$ 

Since we now know that $p_1=p_n$ and all other $p_i=0$, we have from $\sum_{i=0}^n p_i =1$ that $p_1=p_n=\frac{1}{2}$. Similarly, since we now know that $q_1=q_n$ and all other $q_i=\frac{3}{2}q_1$, we have from $\sum_{i=0}^n q_i =1$ that $q_1=q_n=\frac{2}{3n-2}$ and all other $q_i=\frac{3}{3n-2}$, which establishes our claim. \qed

And with our claim established, this completes the proof of Theorem \ref{thm:2-dice}. \end{proof}

\begin{remark}
Our argument shows that the minimum value of $D$, which measures the distance from the uniform distribution on the set $\{2, \ldots, 2n\}$, approaches $0$ as $n\to\infty$. In other words, even though the uniform distribution is impossible to achieve using the sum of two dice, the best approximation using the optimal dice converges to the desired uniform distribution at a rate of $O\left(\frac{1}{n^2}\right)$ (since $D=\frac{1}{2(2n-1)(3n-2)}$) as  $n$, the number of sides, tends to infinity. 
\end{remark}

\begin{remark}
Also regarding convergence, note that if $X$ and $Y$ are independent random variables where $X$ attains the values 0 and 1 with equal probability (of 1/2) and $Y$ is uniformly distributed on the interval $[0, 1]$, then the distribution of $X + Y$ is uniform on $[0, 2]$. This means, from our discrete dice perspective, if we let $X_{n}$ and $Y_{n}$ be random variables corresponding to rolling each of the two optimal dice given in Theorem 1, except that we normalize the die faces to be $1/n,2/n,...,1$, then as $n \to \infty$, $X_{n}$ and $Y_{n}$ converge in distribution to $X$ and $Y$.  This fits with the flavor of our results and helps further indicate why the optimal dice were unlikely to be symmetric (that is, to have $p_{i} = q_{i}$ for all $i=1,2,...,n$).
\end{remark}

\section{An odd result: generating the proof of Theorem \ref{thm:neg-prob}} \label{sect:pqneg}

In this section, we will prove Theorem \ref{thm:neg-prob}, which shows that for $m \ge 2$ dice, allowing negative values for the $p_i, q_i,..., z_i$ makes a uniform distribution for dice sums attainable if and only if $n$ is odd. More precisely, using the notation introduced in Section \ref{sect:notation}, we can rewrite Theorem \ref{thm:neg-prob} as:

\begin{theorem2}
For $m$ dice, allow $p_i,q_i,...,z_i$, where $i=1,2,...,n$, to be any real (possibly negative) numbers, still subject to the restriction that $\sum_{i=1}^n p_i = \sum_{i=1}^n q_i =...=\sum_{i=1}^n z_i =1$. Each $c_j$ is still defined by $c_j=\sum_{i_1 + i_2 +\cdots+i_m=j}p_{i_1}\times q_{i_2}\times\cdots \times z_{i_m}$ for $j=m,m+1,...,mn$.
% $$ c_j =
% \left\{
% \begin{array}{ll}
% \sum_{k=1}^{j-1}p_kq_{j-k}  &\mbox{if } j= 2,3,...,n+1 \\ 
% \sum_{k=0}^{2n-j} p_{n-k}q_{j-n+k}    &\mbox{if } j=n+2,n+3,...,2n,
% \end{array}
% \right.  $$
If and only if $n$ is odd, $p_i, q_i,...,z_i$ can be chosen so that $c_j =\frac{1}{m(n-1)+1}$ for all $j$, and therefore $D=\sum_{j=m}^{mn}\left(c_j -\frac{1}{m(n-1)+1}\right)^2=0$. In other words,  if $n$ is odd, then a uniform distribution for the $c_j$ is possible. If $n$ is even, then a uniform distribution for the $c_j$ is not possible.
\end{theorem2}

% {\bf Extension of Theorem~\ref{thm:neg-prob} to $m$ dice, where $m> 2$:}

\begin{proof}[Proof]If a set of $m$ dice with $n$ sides has uniform totals, then the generating function of the totals is $\frac{1}{m(n-1)+1}(x^m+...+x^{mn}) = \frac{x^m}{m(n-1)+1}T(x)$, where $T(x) = 1+x+\cdots + x^{m(n-1)}$. Note that $T(x)$ has roots at the non-trivial $(m(n-1)+1)$st roots of unity; in particular, it has a single root $-1$ or none according to whether $m(n-1)+1$ is even or odd. Now, the generating function of \emph{each} die is of the form $p_1 x+ p_2 x^{2}+\cdots + p_n x^n = x(p_1+p_2 x+\cdots + p_n x^{n-1})$ for some $p_i$; the second factor must have a real root if $n$ is even. Thus, if $n$ is even, and $m>1$, the product of the dice generating functions cannot equal $T(x)$. If $n$ is odd, then any partition of $m(n-1)/2$ real quadratic factors of $T(x)$ into $m$ sets of size $\frac{n-1}{2}$ yields $m$ generating functions of degree $n-1$ with product $T(x)$. Multiplying each of these polynomials by $x$, and normalizing the polynomials' coefficients appropriately, we obtain $m$ generating functions whose coefficients give us the desired probabilities.  
\end{proof}

\section{Numerical results create two new open questions for multiple dice} \label{sect:multiple-dice}

The numerical experiments conducted by Gasarch and Kruskal ~\cite{Gasarch-Kruskal} for $m=2$ dice inspired us to resolve the open questions they posed. Similarly, our numerical experiments have led to two new open questions for multiple dice that we hope others will explore. Our numerical experiments for $m$ dice were conducted in Python. The Jupyter notebook containing the Python code for the case of three $n$-sided dice is available in an online appendix, which is located at \url{https://tinyurl.com/294ycnha}.\footnote{The code works with probabilities $p_i, q_i,$ and $r_i$ for the three dice (labelled $\mathtt{p\_i}$, $\mathtt{q\_i}$, and $\mathtt{r\_i}$ in the code), where $i = 1, 2,...,n$. The code is easily extended to more than three dice --- or contracted for the case of two dice. To remove the restriction that each $p_i, q_i,$ and $r_i$ must be between 0 and 1, we can simply remove the phrase ``bounds=bounds'' from the minimize command in code block [6]. However, our open questions in this section consider actual probabilities, so we keep the restriction of being between 0 and 1.}

First open question: For two dice, Gasarch and Kruskal suggested that the $n$-sided dice that minimized $D$ had to be {\it symmetric}, meaning for the first die that $p_1=p_n$, $p_2 = p_{n-1}$, $p_3 = p_{n-2}$, etc., and, similarly, for the second die that $q_1=q_n$, $q_2 = q_{n-1}$, $q_3 = q_{n-2}$, etc.  This paper proves that conjecture to be true when $m=2$ dice, but it is an open question as to whether or not this holds for $m>2$ dice. Our numerical experiments always produce symmetric dice results for minimizing $D$, but that is not a proof, of course!

Second open question: In fact, our experimental results go farther than this, indicating a specific symmetric pattern, which is that the quantity $D$ is minimized when the probabilities on one die (say, the first die, so we are specifying the $p_i$) are given by
$$ p_i =
\left\{
\begin{array}{ll}
\frac{m}{(n-2)(2m-1)+2m}  &\mbox{if } i = 1 \mbox{ or } n \\ 
\frac{2m-1}{(n-2)(2m-1)+2m}    &\mbox{if } i=2,3,...,n-1,
\end{array}
\right.  $$
while, for the other $m-1$ dice, we have
$$ q_i = r_i =...=
\left\{
\begin{array}{ll}
\frac{1}{2}  &\mbox{if } i = 1 \mbox{ or } n \\ 
0    &\mbox{if } i=2,3,...,n-1.
\end{array}
\right.  $$
For example, in our online appendix, we see that the results for $m=3$ and $n=5$ (three five-sided dice) fit this pattern. That is,
\begin{align*}
(p_1, p_2, p_3, p_4, p_5) &= \left(\frac{3}{21}, \frac{5}{21},\frac{5}{21},
\frac{5}{21}, \frac{3}{21}\right) \\
(q_1, q_2, q_3, q_4, q_5) &= \left(\frac{1}{2}, 0, 0, 0, \frac{1}{2}\right) \\
(r_1, r_2, r_3, r_4, r_5) &= \left(\frac{1}{2}, 0, 0, 0, \frac{1}{2}\right).
\end{align*}
The above symmetric pattern agrees with the results of every other example we ran as well, using a variety of values for $n$ and $m$, but the proof (or refutation) of this pattern when $m>2$ remains an open question. 

\subsection*{Acknowledgment.} We thank Frank Farris and Chi Hoi Yip for their helpful feedback on the manuscript. We are also grateful to the two anonymous referees who helped us enhance the clarity of our work.


\begin{bibdiv}
\begin{biblist}

\bib{Kelly1}{article}{
 author = {Kelly, John B.},
 journal = {The American Mathematical Monthly},
 number = {6},
 pages = {p.~416},
 title = {Elementary Problem E925},
 volume = {57},
 year = {1950}
}

\bib{Kelly2}{article}{
 author = {Kelly, J. B.},
 author = {Moser, Leo}, 
 author = {Wahab, J. H.},
 author = {Finch, J. V.}, 
 author = {Halmos, P. R.},
 journal = {The American Mathematical Monthly},
 number = {3},
 pages = {191--192},
 title = {Solution to Elementary Problem E925},
 volume = {58},
 year = {1951}
}

\bib{Bol}{book}{
    AUTHOR = {Bollob\'{a}s, B\'{e}la},
     TITLE = {The art of mathematics},
      NOTE = {Coffee time in Memphis},
 PUBLISHER = {Cambridge University Press, New York},
      YEAR = {2006},
     PAGES = {xvi+359},
      ISBN = {978-0-521-69395-0; 0-521-69395-0},

}

\bib{Hof}{book}{
author = {Hofri, Micha},
title = {Analysis of Algorithms},
year = {1995},
publisher = {Oxford University Press},
address = {Oxford, UK}
}

\bib{Hon}{book}{
    author = {Honsberger, Ross}, 
    title = {Mathematical Morsels},
    year = {1978},
    publisher = {The Mathematical Association of America},
    address = {Washington, DC},
}

\bib{Chen-Rao-Shreve}{article}{
 author = {Chen, Guantao},
 author = {Rao, M. Bhaskara},
 author = {Shreve, Warren E.},
 journal = {Mathematics Magazine},
 number = {3},
 pages = {204--206},
 title = {Can One Load a Set of Dice So That the Sum Is Uniformly Distributed?},
 volume = {70},
 year = {1997}
}

\bib{Gasarch-Kruskal}{article}{
 author = {Gasarch, William I.},
 author = {Kruskal, Clyde P.},
 journal = {Mathematics Magazine},
 number = {2},
 pages = {133--138},
 title = {When Can One Load a Set of Dice so That the Sum Is Uniformly Distributed?},
 volume = {72},
 year = {1999}
}

\bib{Mor}{article}{
    AUTHOR = {Morrison, Ian},
     TITLE = {Sacks of dice with fair totals},
   JOURNAL = {Amer. Math. Monthly},
    VOLUME = {125},
      YEAR = {2018},
    NUMBER = {7},
     PAGES = {579--592},
      ISSN = {0002-9890,1930-0972},
   %MRCLASS = {60C05 (11R18 12D05)},
  %MRNUMBER = {3836420},
%MRREVIEWER = {Idris\ David\ Mercer},
       %DOI = {10.1080/00029890.2018.1473699},
       %URL = {https://doi.org/10.1080/00029890.2018.1473699},
}

\end{biblist}
\end{bibdiv}
\end{document}